\documentclass{commat}

\usepackage{tikz}
\usepackage[hang, small,labelfont=bf,up,textfont=it,up]{caption}
\usepackage{graphicx}
\usepackage{subfigure}

\newcommand{\ZZ}{\mathbb Z}
\newcommand{\RR}{\mathbb R}
\DeclareMathOperator{\mean}{mean}
\DeclareMathOperator{\degree}{degree}

\title{%
    Some statistics about tropical sandpile model 
    }

\author{%
    Nikita Kalinin, Yulieth Prieto
    }

\authorinfo[%
    N. Kalinin]{
    Guangdong Technion-Israel Institute of Technology, China}{%
    nikaanspb@gmail.com
    }

\authorinfo[%
    Yu. Prieto]{
    The Abdus Salam International Centre for Theoretical Physics, Italy}{%
    yprieto@ictp.it
    }

\abstract{%
    Tropical sandpile model (or linearized sandpile model) is the only known continuous geometric model exhibiting self-organised criticality. This model represents the scaling limit behavior of a small perturbation of the maximal stable sandpile state on a big subset of $\ZZ^2$. Given a set $P$ of points in a compact convex domain $\Omega\subset \RR^2$ this linearized model produces a tropical polynomial $G_P{\bf 0}_\Omega$. 

Here we present some quantitative statistical characteristics of this model and some speculative explanations. Namely, we study the dependence between the number $n$ of randomly dropped points $P=\{p_1,\dots,p_n\}\subset[0,1]^2=\Omega$ and the degree of the tropical polynomial $G_{P}{\bf 0}_\Omega$. We also study the distributions of the coefficients of $G_{P}{\bf 0}_\Omega$ and the correlation between them. This paper's main (experimental) result is that the tropical curve $C(G_{P}{\bf 0}_\Omega)$ defined by $G_{P}{\bf 0}_\Omega$ is a small perturbation of the standard square grid lines. This explains a previously known fact that most of the edges of the tropical curve $C(G_{P}{\bf 0}_\Omega)$ are of directions $(1,0),(0,1),(1,1),(-1,1)$. 

The main theoretical result is that $C(G_{P}{\bf 0}_\Omega)\setminus (P\cap \partial\Omega)$, i.e. the tropical curve in $\Omega^\circ$ with marked points $P$ removed, is a tree. 
    }



\VOLUME{31}
\NUMBER{3}
\YEAR{2023}
\firstpage{9}
\DOI{https://doi.org/10.46298/cm.10500}

\begin{document}


\section{Abelian sandpile model}

A mathematical object is interesting if it appears in several contexts under different disguises.  The sandpile model evolves by a very simple rule. Hence it is not unexpected that it was discovered independently at least three times: in number theory, combinatorics, and physics. Let us define it.

\begin{enumerate}
    \item Let $\Omega$ be a big compact convex subset of $\RR^2$. Let $\Gamma = \Omega\cap \ZZ^2$, naturally, $\Gamma$ becomes a graph when for each point $(i,j)$ we draw edges from $(i,j)$ to its neighbors $(i+1,j),(i-1,j),(i,j-1),(i,j+1)$. 
\begin{definition}A {\it state} of a sandpile is a function $\phi:\Gamma\to \ZZ_{\geq 0}$, $\phi(v)$ being the number of grains at $v\in\Gamma$. If $\phi(v)\geq 4$, we can {\it topple} $v$ by redistributing four grains from $v$ equally to its four neighbors. Sand falling outside of $\Omega$ disappears, that guarantees that any {\it relaxation} (doing toppling while it is possible) eventually terminates. 
\end{definition}
It is a basic feature of the model, that the result $\phi^\circ$ does not depend on a particular choice of relaxation.

If $\Omega$ is a big rectangle, then the distribution of the sizes of {\it avalanches} (all vertices that topple at least once during a relaxation) caused by subsequent grain dropping at random vertices obeys a {\it power law}, thus the sandpile serves as an example of the self-organized criticality (SOC). That was experimentally observed (and the term SOC was coined) in \cite{BTW} and was proven only recently \cite{bhupatiraju2016inequalities}. A key property of the sandpile is that the recurrent states of the above dynamic correspond to the {\it spanning trees} of the underlying graph $\Gamma$ \cite{Dhar}; hence the distribution of avalanches is related to random spanning forests on $\ZZ^2$, see \cite{MR2077255} for details. Sandpiles were studied extensively in physics literature: their universality \cite{paczuski1996universality,chessa1999universality}, and algebraic properties \cite{Dhar}, and the power spectra of its avalanches \cite{laurson2005power}.

\item Another source of sandpiles is the following problem: given a set of vectors, let us try to minimize their sum with $\pm1$ coefficients. In \cite{spencer1986balancing}, a certain combinatorial game was defined to approach this problem. In \cite{bjorner1991chip},  this game was generalized, and the so-called {\it chip-firing} game was defined: a collection of {\it chips} at the vertices of a finite graph evolves by the same toppling rule: we topple ({\it fire}) a vertex if the number of chips at it is at least its valency. 

\item The set of recurrent configurations of the above dynamic has a group structure, the {\it critical group} of the underlying graph. This observation leads to fruitful connections between sandpiles, Tutte polynomials, and matroid theory \cite{Mer,biggs1999chip}.  One of the results in this direction is the Riemann-Roch theorem for divisors on graphs \cite{MR2355607}, exploiting the similarity between Jacobians of Riemann surfaces and graphs, cf. \cite{MR1478029}. The toppling rule essentially boils down to a {\it discrete Laplacian operator}, and the sandpile group is the critical group of the dual graph of a degeneration of an algebraic curve \cite{MR1019714,lorenzini1991finite} (this study of divisor under a degeneration can be traced back to \cite{Raynaud:1970fv}, Proposition~8.1.2). 
\end{enumerate}

We recommend the following introductory reading about sandpiles: \cite{MR2246566,ivashkevich1998introduction,redig,jarai2014sandpile,divsand}.


\subsection{Self-reproducing patterns and tropical curves}
 Patterns in sandpiles, resembling tropical curves, were studied by S. Caracciolo, G. Paoletti, and A. Sportiello, \cite{firstsand}. Later, the corresponding dynamic after the scaling limit was rigorously defined under the name of ``tropical sandpile model'' by N. Kalinin and M. Shkolnikov, whose main result is that the deviation set of $\phi^\circ$ (where $\phi = 3+\sum_{p\in P} \delta_{p}$, $P\subset \Gamma$) is a tropical analytic curve that passes through $P$; see\cite{us} (for a short version read \cite{announce}). Recently is has been shown that the tropical sandpile model exhibits self-organised criticality behaviour \cite{sandcomputation}.

\begin{figure}[h!]
\centering
\subfigure{\includegraphics[width=0.23\textwidth]{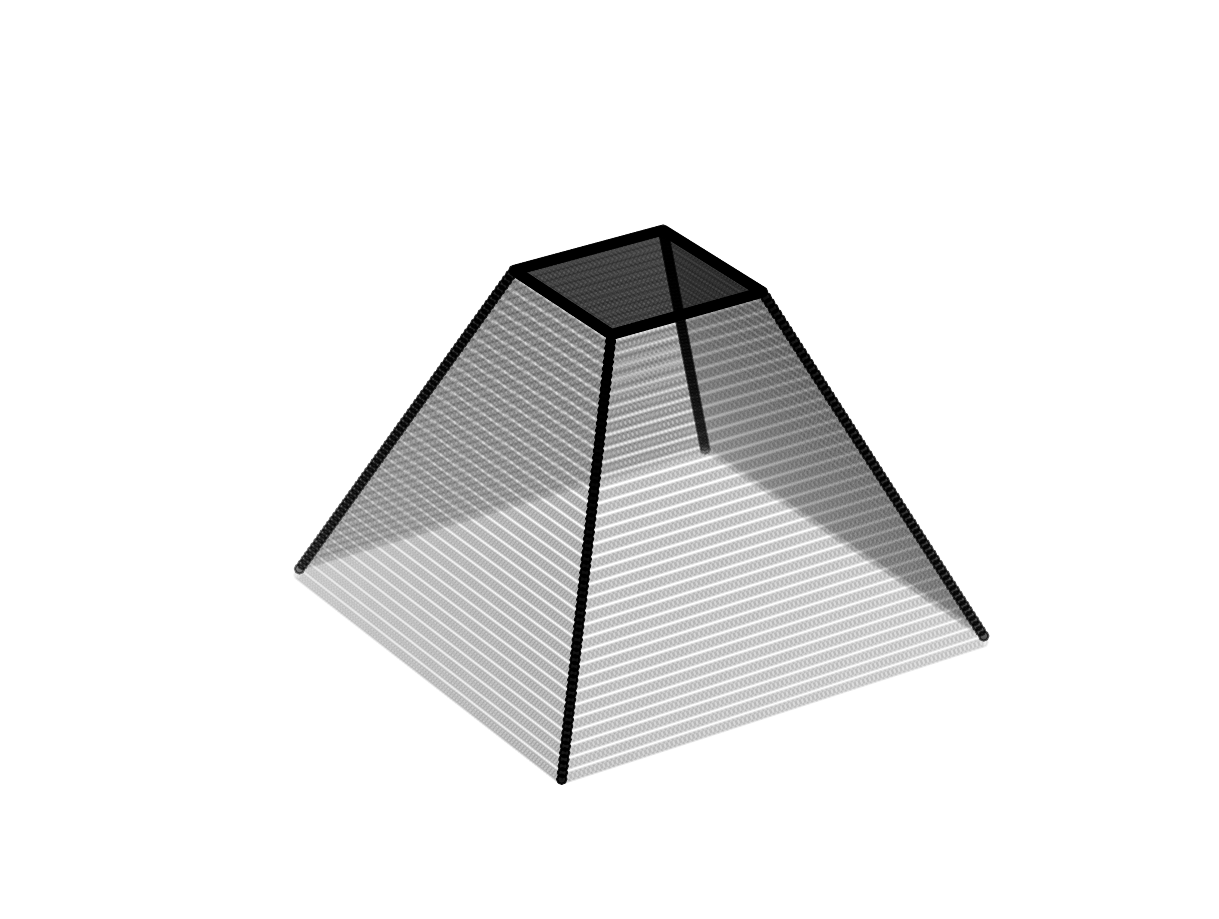}}
\subfigure{\includegraphics[width=0.23\textwidth]{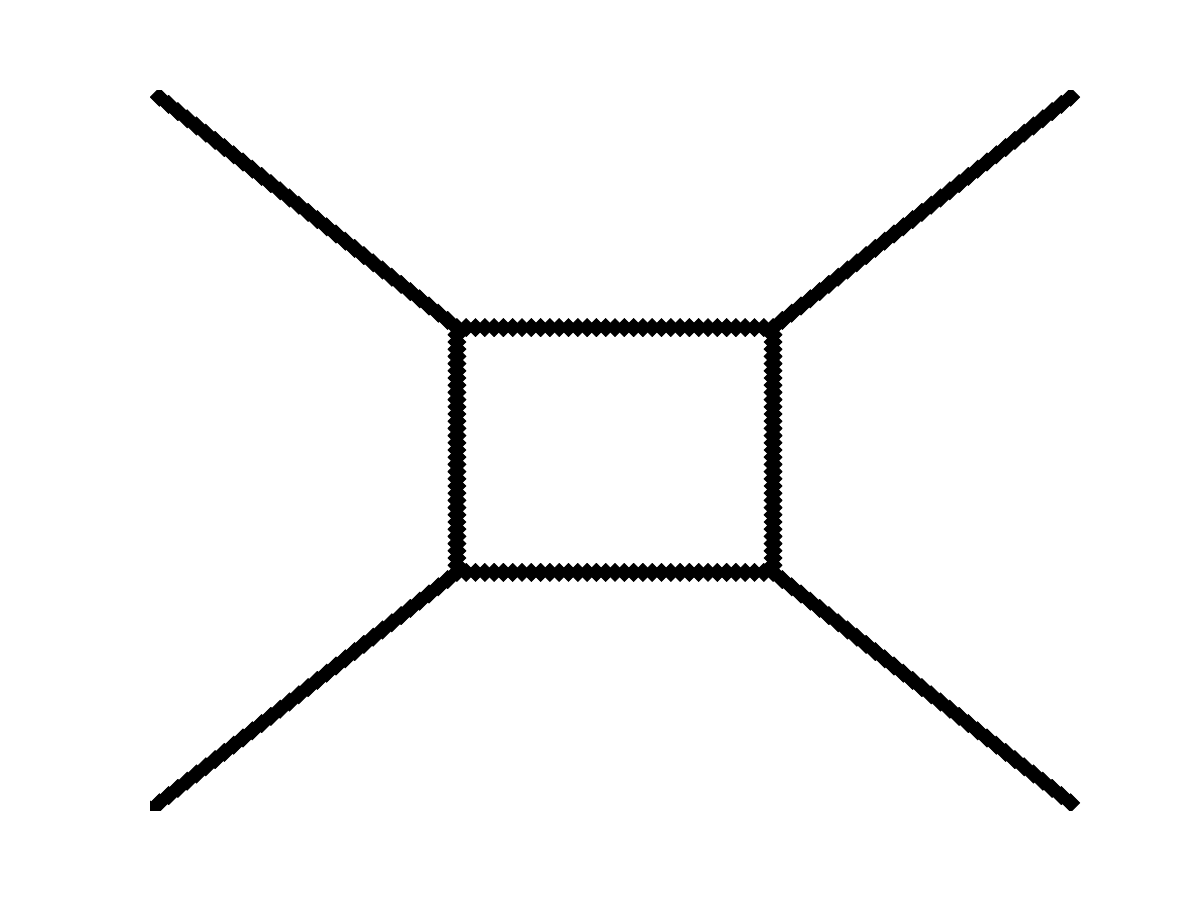}}
\caption{$\Omega$-tropical curve $C(f)$ corresponding to
$f(x,y)=\min(1/3,x,y,1-x,1-y)$.}
\label{ex:1}
\end{figure}

\section{The tropical sandpile model}

Tropical curves are graphs of a special type appearing as degenerations of Riemann surfaces (see also their connections with string theory \cite{tourkine2017tropical} and statistical physics \cite{angelelli2015tropical}), and are used to compute Gromov-Witten invariants, see introductory texts \cite{MR2902202,BIMS}. Here we need only the simplest version of a tropical curve, a planar tropical curve. 
Let $\Omega \subset \mathbb R ^2$ be a compact convex set.
\begin{definition}
An $\Omega$-{\it tropical series} is a function $f:\Omega\to\mathbb R$ such that $f \mid_{\Omega}\geq 0, f\mid_{\partial\Omega}= 0$ and there exist $c_{ij}\in\mathbb R$ such that for each point $ (x,y)\in \Omega^\circ$ we have
\begin{equation*}
\label{eq_series}
f(x,y)=\min\{c_{ij}+ix+jy| (i,j)\in \mathbb Z^2\}. 
\end{equation*}

The set of non-smooth points of $f$ is called an $\Omega$-{\it tropical analytic curve} and is denoted by $C(f)$.
\end{definition}

It is easy to see that the complement of $C(f)$ in $\Omega^\circ$ is divided into parts where only one monomial $c_{ij} + ix +jy$ is minimal.

\begin{definition}
The minimal canonical form of an $\Omega$-tropical series is the unique presentation with the minimal by inclusion set of monomials and minimal possible coefficients $c_{ij}$.
\end{definition}


Let $V(\Omega)$ be the set of all $\Omega$-tropical series. To each finite subset $P \subset \Omega^\circ$ we associate an operator $G_P: V(\Omega) \rightarrow V(\Omega)$ which associates to any $\Omega$-tropical curve and a set $P$ of points a new $\Omega$-tropical curve passing through $P$ as follows.

\begin{definition} Let $G_Pf(x,y)=\min \{g(x,y) \mid  g \in V(\Omega), \ g\geq f \textit{ and } P \subset C(g) \}.$
\end{definition}

Clearly, if $P \subset C(f)$ then $G_P f=f$. Suppose that $P = \{p\}$. Then, if $p$ belongs to a face where the monomial $c_{ij} + ix +jy$ is minimal, then $G_p$ just increases the coefficient $c_{ij}$, \cite{us_series}. On the level of associated $\Omega$-tropical curves, $G_p$ contracts the face which $p$ belongs to, see Figure \ref{fig_ShrinkPhi}.

\begin{figure*}[h]
    \centering

\begin{tikzpicture}[scale=0.4]
\draw[very thick](0,0)--++(1,1)--++(0,7)--++(4,0)--++(1,-1)--++(0,-6)--++(-5,0);
\draw[very thick](1,8)--++(-1,1);
\draw[very thick](5,8)--++(0,1);
\draw[very thick](6,7)--++(1,0);
\draw[very thick](6,1)--++(1,-1);
\draw(3,4)node{$\bullet$};
\draw(3,4)node[right]{$p$};
\draw(4.5,6.5)node{\Large$\Phi$};
\draw[->][very thick](8,4)--(10,4);

\begin{scope}[xshift=300]
\draw[very thick](0,0)--++(2,2)--++(0,5)--++(3,0)--++(0,-5)--++(-3,0);
\draw[very thick](2,7)--++(-2,2);
\draw[very thick](5,7)--++(0,2);
\draw[very thick](5,7)--++(2,0);
\draw[very thick](5,2)--++(2,-2);
\draw(3,4)node{$\bullet$};
\draw(3,4)node[right]{$p$};
\draw[->][very thick](8,4)--(10,4);
\end{scope}

\begin{scope}[xshift=600]
\draw[very thick](0,0)--++(3,3)--++(0,3)--++(1,0)--++(0,-3)--++(-1,0);
\draw[very thick](3,6)--++(-3,3);
\draw[very thick](5,7)--++(0,2);
\draw[very thick](5,7)--++(2,0);
\draw[very thick](5,7)--++(-1,-1);
\draw[very thick](4,3)--++(3,-3);
\draw(3,4)node{$\bullet$};
\draw(3,4)node[left]{$p$};
\end{scope}

\end{tikzpicture}
\caption{The operator $G_p$ shrinks the face $\Phi$ where $p$ belongs to. Note that combinatorics of the new curve can change when shrinking.} 
\label{fig_ShrinkPhi}
\end{figure*}
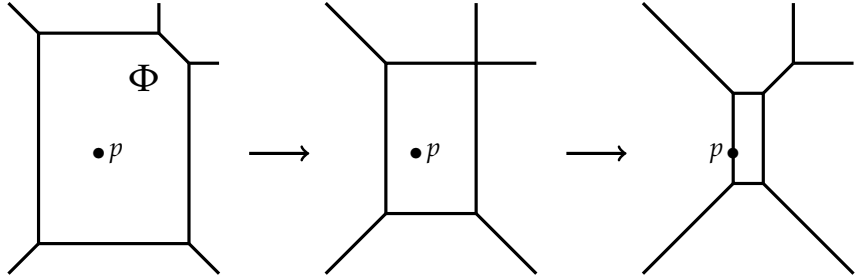

The operators $G_P$ appear in works of C. Vafa  under the name of ``breathing mode'', see \cite{vafa2012supersymmetric}.

The theory of tropical series, operators $G_P$, and all proofs of the aforementioned results can be found in \cite{us_series}. The $\Omega$-tropical curves recently appear in the study of Lagrangian submanifolds \cite{matessi2018lagrangian,mikhalkin2018examples,hicks2019tropical}. An introduction to tropical geometry can be found in \cite{BIMS,mikh2,maclagan2015introduction}.

Dynamics of $G_p$ in the case of one-dimensional tropical series was thoroughly studied in \cite{shkol}.

\section{Experiments and results}

We drop $n$ random uniformly chosen points $P=\{p_1,\dots,p_n\}$ to $\Omega=[0,1]^2$. In order to facilitate the computations, we fix $s\in\mathbb N$ and choose $n$ random points $p_i=(\frac{x}{s},\frac{y}{s})$ where $x,y\in [0,s]\cap \ZZ$ and study the $\Omega$-tropical series $G_P{\bf 0}_{[0,1]^2}$, i.e. $G_P$ applied to the identically zero function on $\Omega$. We assume that when $s$ is big enough, this model is close to the continuous model. 

Let
\begin{eqnarray*}
S &:=& \{ 50,100,200,300,500,1000 \},\\
N &:=& \{ 100,500,1000,1500,2000,2500, 5000 \},
\end{eqnarray*}
For every $s \in S$ and $n \in N$, we run experiments as explained above and study the following characteristic of $G_P0_\Omega$. 

\begin{figure}[h]
\centering
\subfigure{\includegraphics[width=0.33\textwidth]{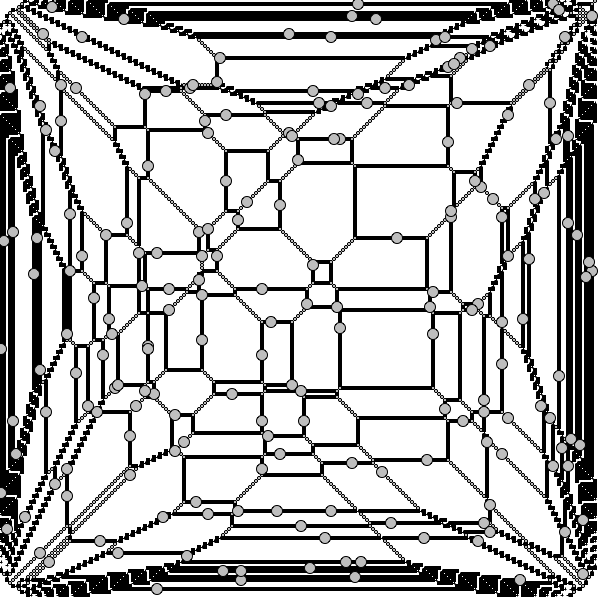}}
\caption{A typical example of $G_P{\bf 0}_{[0,1]^2}$.}
\label{ex:2}
\end{figure}

%

\subsection{The minimal degree of \texorpdfstring{$G_P0$}{GP0}}

\begin{definition}
If $f(x,y)=\min\limits_{(i,j)\in \mathcal{A}}(c_{ij}+ix+jy)$, then the (tropical) degree of $f$ is defined as
\begin{equation*}
\mathtt{Degree}(f)=\max\limits_{(i,j) \in \mathcal{A}} (|i|+|j|).
\end{equation*}
\end{definition}
This degree is the analog of the degree of a polynomial in two variables in algebraic geometry. It is known that a curve $\{F(x,y) = 0\}$ where $F$ is a polynomial in two variables of degree $d$ can be drawn through $\frac{(d+1)(d+2)}{2}-1$ generic points (e.g. a line through two points, a conic through five points). So it is reasonable to expect that $\mathtt{Degree}(G_{P}0_{[0,1]^2})$ should be of order $\sqrt{|P|}=\sqrt n$.

According to simulations for $S,N$ as above, we numerically observe that the families $f_{\texttt{Degree}}^{\min}$ and $f_{\texttt{Degree}}^{\texttt{mean}}$ defined as follows 
\begin{eqnarray*}
f_{\texttt{Degree}}^{\min}(n) &=& \frac{\min (\degree(s,n))}{n^{\frac{1}{2}-\epsilon}} \\
f_{\texttt{Degree}}^{\texttt{mean}}(n)&=& \frac{\mean (\degree(s,n))}{n^{\frac{1}{2}-\epsilon}}
\end{eqnarray*}
converge to constants as $s$ and $n$ tend to infinity, and $\epsilon$ is a small number, see Figure \ref{im:degree}.

\begin{figure}[h!]
\centering
\includegraphics[width=0.43\textwidth]{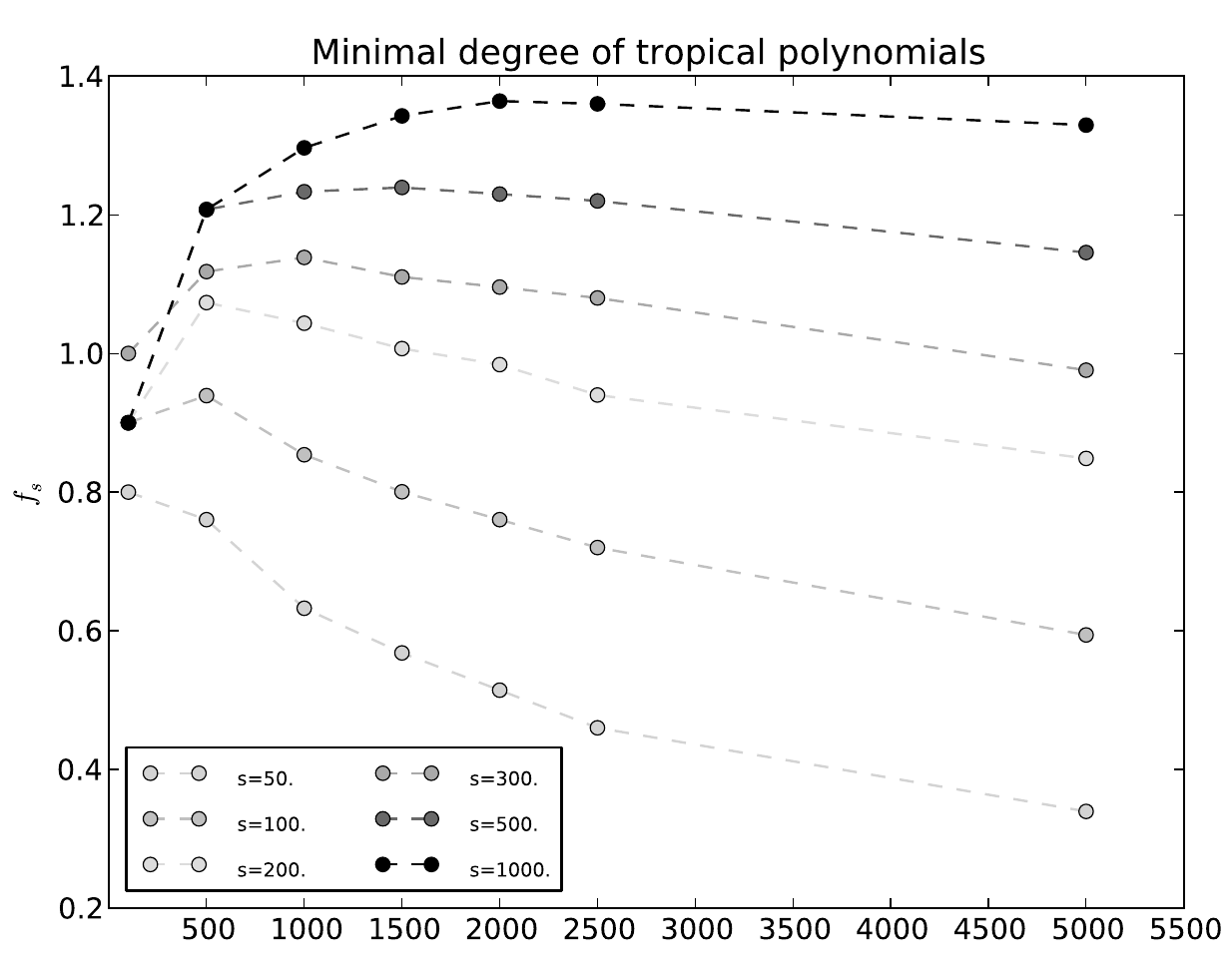}
\includegraphics[width=0.43\textwidth]{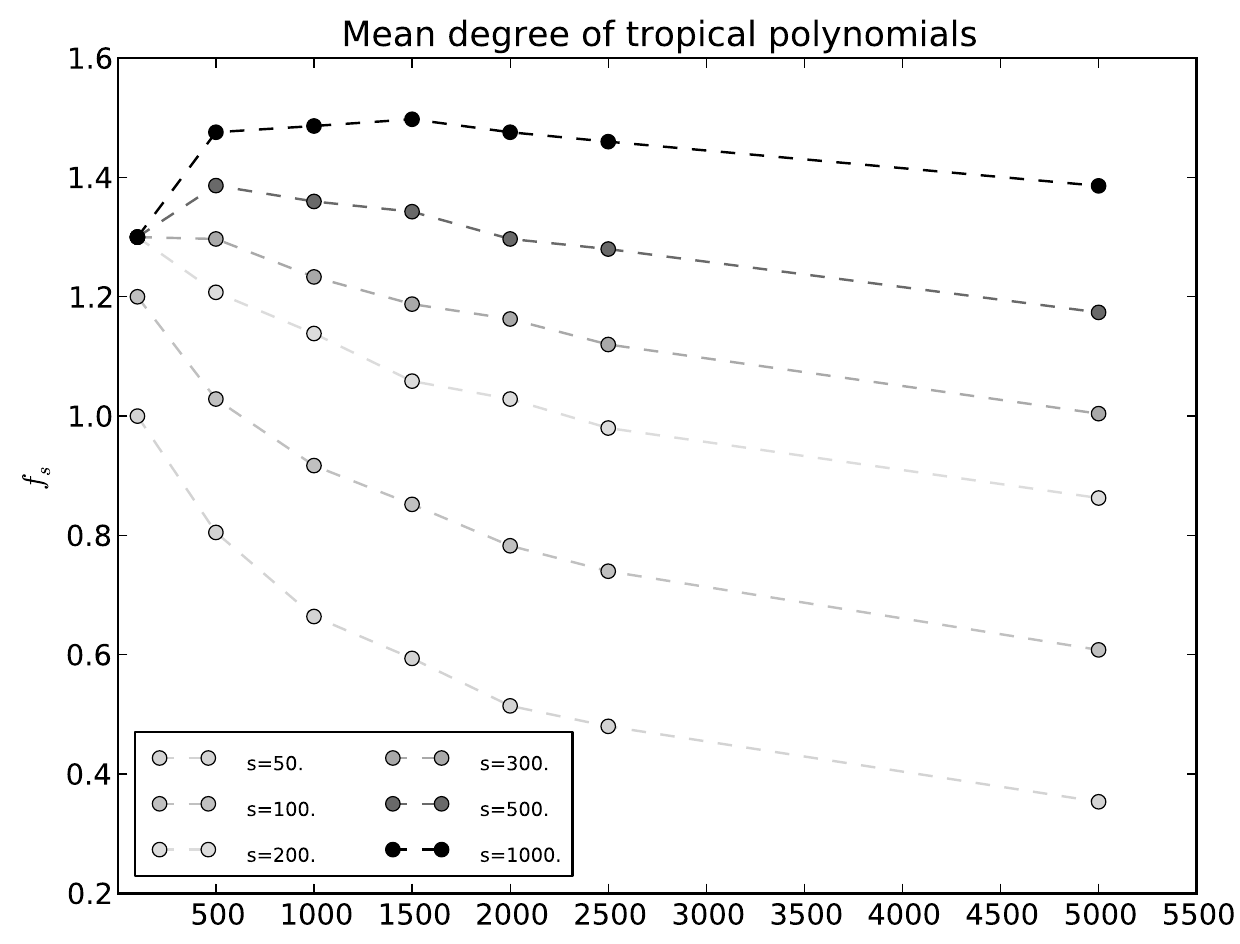}
\caption{Minimum and mean degree of $5000$ experiments for each $s \in S$ (different curves) and $n \in N$ ($x$-axe).}
\label{im:degree}
\end{figure}

In other words, the limit of the degree distribution depends only on the number of randomly chosen points and does not depend on the size of square if it is large enough. Moreover, most of the degrees are approximately $\sqrt{n}$, so we conjecture a kind of concentration of measure (the standard deviation of the distribution of degrees tends to zero).

\subsection{Genus}
As in the classical algebraic geometry we can define the genus of a curve. Together with the degree, these are two main invariants of plane curves.
\begin{definition}
The genus of an $\Omega$-tropical curve $C(f)$ is the number of connected components of $\Omega^\circ\setminus C(f)$ whose closure does not intersect $\partial\Omega$.
\end{definition}

For example, the genus of the curve on Figure~\ref{ex:2} is one, as well as the genera of the curves on Figure~\ref{fig_ShrinkPhi}. In classical algebraic geometry, the genus represents the dimension of the deformations of a curve in a given class of curves. Thus, in our setting where the dimension of deformations of a curve $C(G_P0_\Omega)$ is naturally the number $|P|$ of points, we would expect that the genus of $G_P0_\Omega$ is equal to $P$. Indeed, this is the case for a generic collection $P$ of points.

\begin{theorem}
If $P$ is a generic collection of points in $\Omega^\circ$, then the genus of $C(f)$ (where $f=G_P0_\Omega$) is equal to $P$.
\end{theorem}

\begin{proof} Firstly, note that every connected component of $\Omega^\circ\setminus C(f)$, whose boundary does not intersect $\partial\Omega$, contains a point $p\in P$ in its closure. Indeed, if it is not the case, consider the monomial $a_{ij}+ix+jy$ which is minimal on this connected component and note that $a_{ij}$ could be decreased by a small amount, and the corresponding tropical series would still contain $P$ in its corner locus and be greater than zero which contradicts the minimality of $f=G_P0_\Omega$.

Secondly, because of the genericity of the set $P$, we may assume that a small perturbation of points in $P$ does not change the combinatorial type of the tropical curve. Consider $X=C(f)\setminus (P\cup \partial\Omega)$. It does not contain cycles by the above arguments, and to prove that $X$ is a tree we only need to establish its connectivity. Suppose the contrary and take any connected component of $X$: $X$ is a tree with endpoints on the boundary of $\Omega$ and in $P$. Consider any point $p\in P$ which is contained in the closure of $X$. Note that if we perturb $P$ by moving $p$ orthogonally to the edge containing it, $X$ should be also perturbed. This means that $X$ must contain both connected components of the intersection of a small neighborhood of $p$ with $C(f)$. Otherwise it would be possible to slightly move one of the legs of a tropical curve without moving any other leg (which would contradict the tropical Menelaus theorem \cite{Mikhalkin:2015kq} also known as the tropical momentum theorem \cite{guide}, or the moment condition \cite{Yoshitomi:kq}). Thus the closure of $X$ is a connected component of $C(f)$ and recall that $C(f)$ is connected. Therefore $C(f)\setminus (P\cup \partial\Omega)$ is connected from the very beginning, hence it is a tree, showing that $|P|$ is equal to the genus of $C(f)$.
\end{proof}

\subsection{Behaviour of coefficients \texorpdfstring{$c_{ij}$}{cij}}
Let $c_{ij}$ be the coefficient of $G_P0_{[0,1]^2}$, as in \eqref{eq_series}.
Let $\{ f_s^{ij} \}_{s \in S}$ be
\begin{equation*}
f_s^{ij}(n):= \frac{\mean ({c_{ij}}(s,n))}{sn^{\alpha}},
\end{equation*}
where $\alpha$ is a scaling parameter that we want to estimate, and $c_{ij}$ is one of $\{ c_{00}$, $c_{11}$, $c_{10}$, $c_{01} \}$. We numerically observe that $f_s^{ij}(n)$ converges to a constant function when $\alpha=1/2$ and $n,s$ tend to infinity. More plots can be found in \cite{yulieth}.

\subsection{Mean identity}

There exists a nice relation  between coefficients $c_{00}$,  $c_{11}$, $c_{10}$ and $c_{01}$.  Namely, writing $\overline c_{00}$ for the average of $c_{00}$ in the experiments, we obtain that
\begin{equation*}
\label{eq_simetrycenter}
\overline{c_{00}}(s,n)+\overline{c_{11}}(s,n)=\overline{c_{10}}(s,n)+\overline{c_{01}}(s,n).
\end{equation*}
is satisfied, see Figure \ref{im:Levine_id}. Note that it should be of order $\sqrt n$ a priori as all of its summands.
\begin{figure}[h]
\centering
\includegraphics[width=0.5\textwidth]{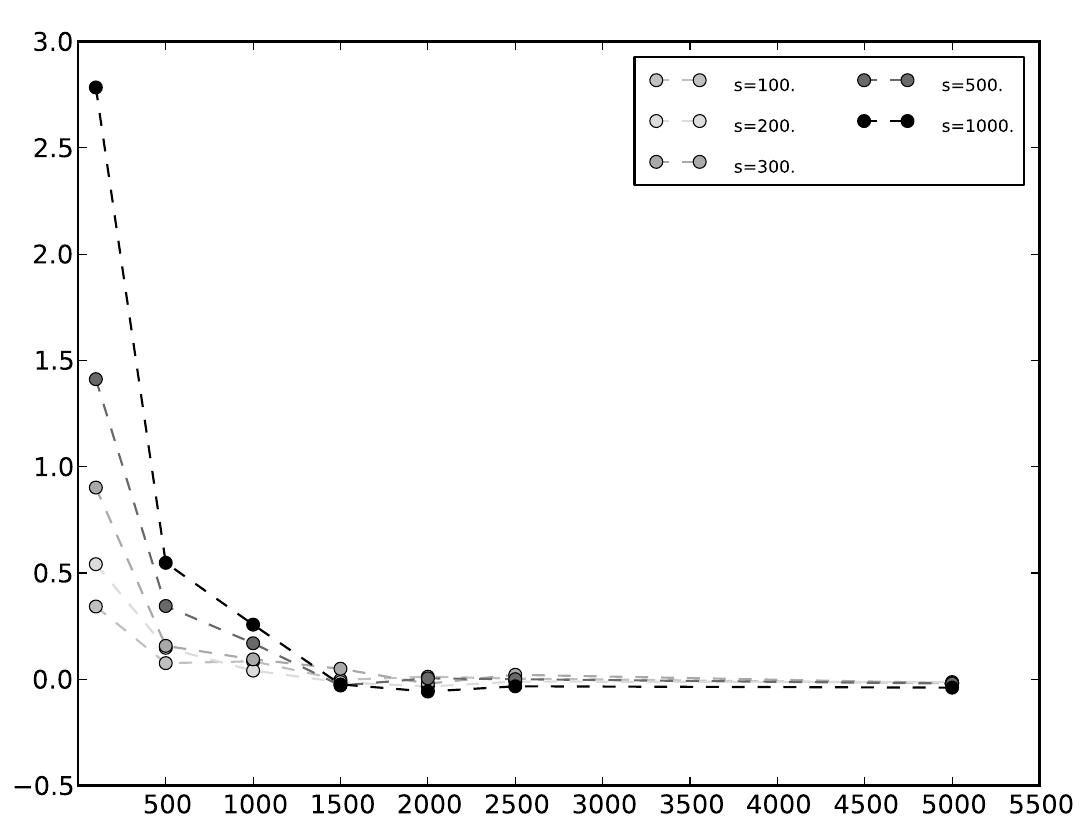}
\caption{Graph of $\overline{c_{00}}(s,n)+\overline{c_{11}}(s,n)-\overline{c_{10}}(s,n)-\overline{c_{01}}(s,n)$.}
\label{im:Levine_id}
\end{figure}

The $1/n$ of the standard deviation of values $c_{ij}$ is constant, see Figure \ref{im:sdcoo} for $f_s^{SD}(n)=\frac{SD(c_{00})}{sn}.$

\begin{figure}[h!]
\centering
    \includegraphics[width=0.5\textwidth]{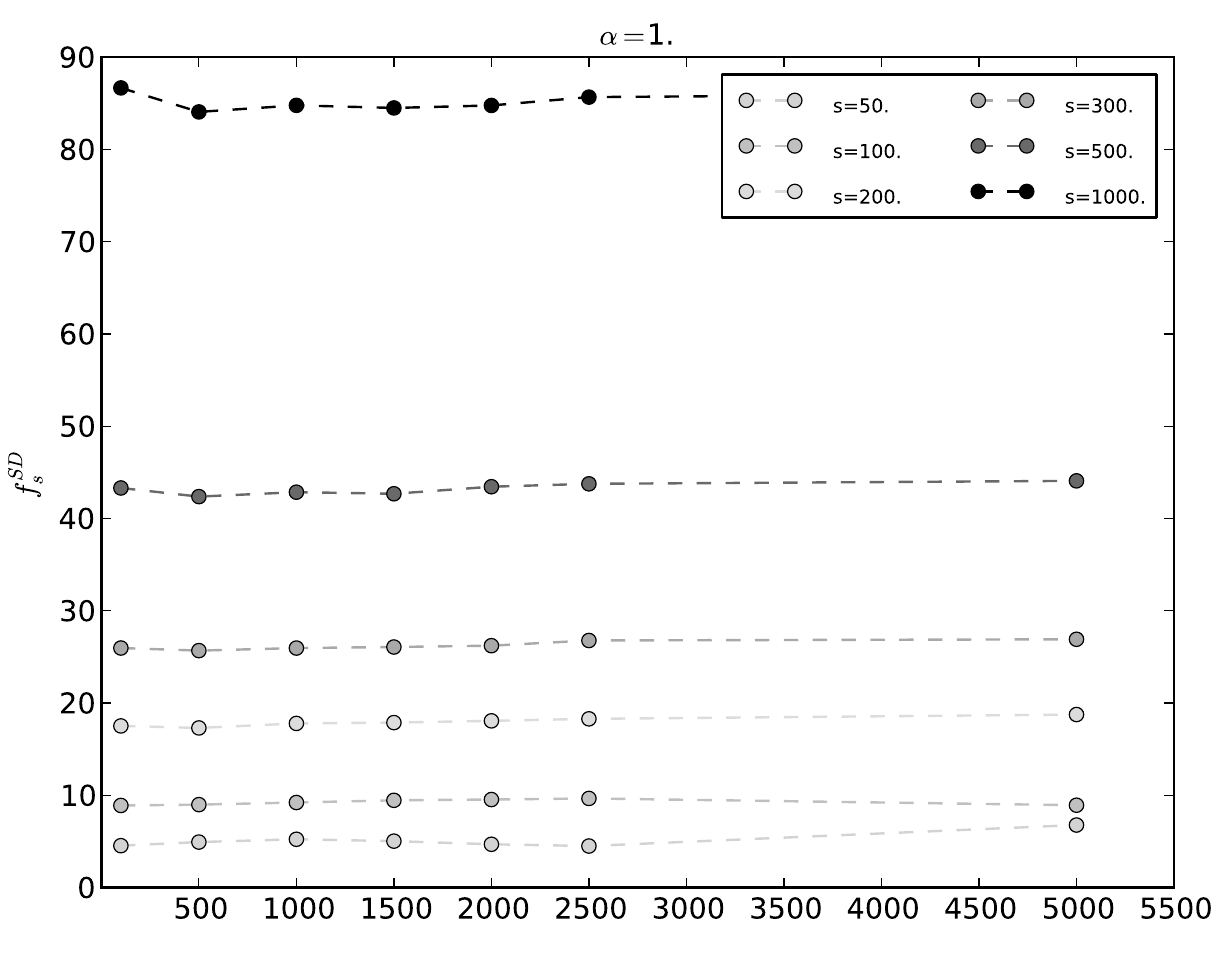}
\caption{Standard deviation for values of $c_{00}$.}
\label{im:sdcoo}
\end{figure}
\section{Discussion}

We conjecture that the general identity
\begin{equation}
\label{eq_simetrycenter2}
\overline{c_{ij}}+\overline{c_{i+1j+1}}=\overline{c_{i+1j}}+\overline{c_{ij+1}}.
\end{equation}
for the averages of the coefficients is satisfied when $i^2 + j^2$ is relatively small with respect to $\sqrt{n}$. This equality suggests that the tropical curve is a small perturbation of a square lattice (whose tropical equation satisfy \eqref{eq_simetrycenter2}).  But, believing that the tropical curve is of degree $d=\sqrt n$ and is almost a grid we can write that $c_{0,-d} = 0$ (because the tropical series is zero on the left side of $[0,1]^2$ and is equal to $c_{0,-d}-dx$ there) and $c_{0,k-d}=c_{0,k+1-d}+\frac{k}{2d}$, because  near $x=\frac{k}{2d}$ we have $$c_{0,k-d}+(k-d)x=c_{0,k+1-d}+(k+1-d)x.$$  Therefore $c_{00}= \frac{(d+1)d}{2d}\sim d=\sqrt n$, which agrees with the experimental evidence.  

Our simulations suggest that the average degree of an $\Omega$-tropical curve through $n$ generic points is approximately $\sqrt n$, but a naive argument to show that by random dropping  of points we get one point per small square fails, see \cite{law}.

\subsection*{Acknowledgements}

We thank to Aldo Guzman who wrote all the code in C++, Mikhail Shkolnikov for several helpful comments concerning during the preparation of the paper and Ernesto Lupercio for the enthusiasm and support. We thank the service teams of Xiuhcoatl and Abacus supercomputers who dealt with technicalities. The computer simulations were performed on Abacus, in the Applied Mathematics Computational Laboratory of High Performance at Mexico, in 2016. 

{\small\bibliography{stb}}

\EditInfo{December 17, 2022}{June 2, 2023}{Jacob Mostovoy and Sergei Chmutov}

\end{document}